\documentclass{article}
\usepackage[margin=1in]{geometry}

\title{The $a$-numbers of Jacobians of Suzuki curves}
\author{Holley Friedlander\and Derek Garton\and Beth Malmskog\and Rachel Pries\and Colin Weir}

\usepackage{amsfonts}
\usepackage{amsmath}
\usepackage{amsthm}
\usepackage{amssymb}
\usepackage{url, verbatim}

\DeclareMathOperator{\Jac}{Jac}

\DeclareMathOperator{\Hom}{Hom}
\DeclareMathOperator{\car}{\mathcal{C}}
\DeclareMathOperator{\ch}{char}
\newcommand{\F}{\mathbb{F}}
\newcommand{\Hnot}{H^0(S_m,\Omega^1)}

\newtheorem{theorem}{Theorem}[section]
\newtheorem{lemma}[theorem]{Lemma}
\newtheorem{proposition}[theorem]{Proposition}

\newtheorem{predefinition}[theorem]{Definition}

\newtheorem{preremark}[theorem]{Remark}
\newenvironment{remark}{\begin{preremark}\rm}{\end{preremark}}
\newtheorem{prenotation}[theorem]{Notation}

\newtheorem{preexample}[theorem]{Example}
\newenvironment{example}{\begin{preexample}\rm}{\end{preexample}}
\newtheorem{preclaim}[theorem]{Claim}

\newtheorem{prequestion}[theorem]{Question}
\newenvironment{question}{\begin{prequestion}\rm}{\end{prequestion}}
 \makeatletter
\def\emppsubsection{\@startsection{subsection}{2}{\z@}{-3.25ex plus -1ex minus -.2ex}{-1em}{\bf}}

\begin{document}

\maketitle

\begin{abstract}
For $m \in {\mathbb N}$, let $S_m$ be the Suzuki curve defined over $\F_{2^{2m+1}}$. 
It is well-known that $S_m$ is supersingular, but the $p$-torsion group scheme of its Jacobian is not known.
The $a$-number is an invariant of the isomorphism class of the $p$-torsion group scheme.
In this paper, we compute a closed formula for the $a$-number of $S_m$ using the action of the Cartier operator on $\Hnot$.
\end{abstract}

\section{Introduction} \label{Introduction}

Let $m\in {\mathbb N}$, $q=2^{2m+1}$, and $q_0=2^m$. 
The \emph{Suzuki curve} $\mathcal{S}_m\subset\mathbb{P}^2$ is defined over $\F_q$
by the homogeneous equation:
\[
W^{q_0}(Z^q + ZW^{q-1}) = Y^{q_0}(Y^q + YW^{q-1}).
\]
This curve is smooth and irreducible with genus $g=q_0(q-1)$ and it has exactly one point at infinity  \cite[Proposition 1.1]{HansenStich}.
The number of points on the Suzuki curve over $\F_q$ is $\#S_m\left(\F_q\right)=q^2+1$; this number 
is optimal in that it reaches Serre's improvement to the Hasse-Weil bound \cite[Proposition 2.1]{HansenStich}. 

In fact, $S_m$ is the unique $\mathbb{F}_q$-optimal curve of genus $g$ \cite{FTunique}.
This shows that $S_m$ is the Deligne-Lusztig variety of dimension $1$ associated with the group $Sz(q)= {}^2B_2(q)$ 
\cite[Proposition 4.3]{H92}.
The curve $S_m$ has the Suzuki group $Sz(q)$ as its automorphism group; the order of $Sz(q)$ is $q^2(q-1)(q^2+1)$ 
which is very large compared with $g$.
Because of the large number of rational points relative to their genus, the Suzuki curves provide good examples of Goppa codes 
\cite[Section 4.3]{GK08}, \cite{GKT}, \cite{HansenStich}. 

The $L$-polynomial of $S_m$ is $(1+\sqrt{2q} t +q t^2)^g$ \cite[Proposition 4.3]{H92}.
It follows that $S_m$ is supersingular for each $m \in {\mathbb N}$.  
This fact implies that the Jacobian ${\rm Jac}(S_m)$ is isogenous to a product of 
supersingular elliptic curves and that ${\rm Jac}(S_m)$ has no $2$-torsion points over $\overline{\F}_2$.  
However, there are still open questions about ${\rm Jac}(S_m)$.
In this paper, we address one of these by computing a closed formula for the $a$-number of ${\rm Jac}(S_m)$.   

The $a$-number is an invariant of the $2$-torsion group scheme ${\rm Jac}(S_m)[2]$.
Specifically, if $\alpha_{2}$ denotes the kernel of Frobenius on the additive group $\mathbb{G}_{a}$, 
then the $a$-number of $S_m$ is $a(m) = {\rm dim}_{\overline{\F}_2}\text{Hom}(\alpha_{2},\text{Jac}(S_m)[2])$. 
It equals the dimension of the intersection of ${\rm Ker}(F)$ and ${\rm Ker}(V)$ on the 
Dieudonn\'e module of ${\rm Jac}(S_m)[2]$.
Having a supersingular Newton polygon places constraints upon the $a$-number but does not determine it.
The $a$-number also gives partial information about the decomposition of ${\rm Jac}(S_m)$ 
into indecomposable principally polarized abelian varieties, Lemma~\ref{decompose},
and about the Ekedahl-Oort type of ${\rm Jac}(S_m)[2]$, see Section~\ref{Corollaries}.

In Section~\ref{Derek}, we prove that the $a$-number of $S_m$ is $a(m)=q_0(q_0+1)(2q_0+1)/6$, see Theorem~\ref{maintheorem}.
The proof uses the action of the Cartier operator on $\Hnot$ as computed in Section~\ref{Colin}.

Author Pries was partially supported by NSF grant DMS-11-01712. 
We would like to thank the NSF for sponsoring the research workshop for graduate students at Colorado State University in June 2011 
where the work on this project was initiated.
We would like to thank Amy Ksir and the other workshop participants for their insights.

\section{The $a$-number} \label{preliminaries}

Suppose $A$ is a principally polarized abelian variety of dimension $g$ defined over an algebraically closed field $k$ of characteristic $p > 0$.
For example, $A$ could be the Jacobian of a $k$-curve of genus $g$.
Consider the multiplication-by-$p$ morphism $[p]:A \to A$ which is a finite flat morphism of degree $p^{2g}$.
It factors as $[p]=V \circ F$.  Here, $F:A \to A^{(p)}$ is the relative Frobenius morphism 
coming from the $p$-power map on the structure sheaf; it is purely inseparable of degree $p^g$.  
The Verschiebung morphism $V:A^{(p)} \to A$ is the dual of $F$.  

The kernel of $[p]$ is $A[p]$, the $p$-torsion of $A$, which is a quasi-polarized $BT_1$ group scheme.  
In other words, it is a quasi-polarized finite commutative group scheme annihilated by $p$, again having 
morphisms $F$ and $V$.  The rank of $A[p]$ is $p^{2g}$.
These group schemes were classified independently by
Kraft (unpublished)~\cite{Kraft} and by Oort~\cite{O:strat}.
A complete description of this topic can be found in~\cite{M:group} or \cite{O:strat}.

Two invariants of (the $p$-torsion of) an abelian variety are the $p$-rank and $a$-number.
The \emph{$p$-rank} of $A$ is $r(A)=\dim_{\mathbb{F}_p}\left(\Hom\left(\mu_p,A[p]\right)\right)$,
where $\mu_p$ is the kernel of Frobenius on the multiplicative group $\mathbb{G}_m$.
Then $p^{r(A)}$ is the cardinality of $A[p]\left(\overline{\mathbb{F}}_p \right)$.
The \emph{$a$-number} of $A$ is $a(A)=\dim_k\left(\Hom\left(\alpha_p, A[p]\right)\right)$,
where $\alpha_p$ is the kernel of Frobenius on the additive group $\mathbb{G}_a$.
It is well-known that $1 \leq a(A) +r(A) \leq g$. 
Another definition for the $a$-number is 
\[a(A)={\rm dim}_{{\mathbb F}_p}({\rm Ker}(F) \cap {\rm Ker}(V)).\]

If $X$ is a (smooth, projective, connected) $k$-curve, then the $a$-number of $A={\rm Jac}(X)$ 
equals the dimension of the kernel of the Cartier operator $\car$ on $H^0(X,\Omega^1)$ \cite[5.2.8]{LO}.
The reason for this is that the action of $\car$ on $H^0(X,\Omega^1)$ is the same as the action of $V$ on $V{\rm Jac}(X)[p]$.
This is the property that we use to calculate the $a$-number $a(m)$ of the Jacobian of the Suzuki curve $S_m$.

\section{Regular 1-forms for the Suzuki curves} \label{Colin}

In this section, we compute the action of the Cartier operator on the 
vector space of regular 1-forms for the Suzuki curves.

\subsection{Geometry of the Suzuki curves}

Let $m\in {\mathbb N}$, $q=2^{2m+1}$, and $q_0=2^m$. 
Consider the Suzuki curve $\mathcal{S}_m\subset\mathbb{P}^2$ defined over $\F_q$
by the homogeneous equation:
\[
W^{q_0}(Z^q + ZW^{q-1}) = Y^{q_0}(Y^q + YW^{q-1}).
\]

The curve $S_m$ is smooth and irreducible and has one point $P_\infty$ at infinity 
(when $W=Y=0$ and $Z=1$).
Consider the irreducible affine model of $S_m$ defined by the equation
\begin{equation}\label{Suzuki}
z^q + z = y^{q_0}(y^q + y)
\end{equation}
where $y:=Y/W$ and $z := Z/W$.

The following result is well-known, see e.g., \cite[Proposition 1.1]{HansenStich}.
We include an alternative proof that illustrates the geometry of some of the quotient curves of $S_m$
and an important point about the $a$-number.

\begin{lemma} \label{Lgenus}
The curve $S_m$ has genus $g=q_0(q-1)$.
\end{lemma}

\begin{proof}
The set $\F_q^*=\left\{\mu_1, \ldots, \mu_{q-1}\right\}$ can be viewed as a set of representatives for the 
$q-1$ cosets of ${\mathbb F}^*_2$ in ${\mathbb F}^*_q$.
The Suzuki curve has affine equation $z^q-z=f(y)$ where $f(y)=y^{q_0+q}+ y^{q_0+1} \in \F_2(y)$.
For $1 \leq i \leq q-1$, let $Z_{i}$ be the Artin-Schreier curve with equation 
$z_i^2-z_i = \mu_i f(y)$.
As seen in \cite[Proposition 1.2]{GS91}, the set $\{Z_i \to {\mathbb P}^1_y \ \mid \ 1 \leq i \leq q-1\}$ is exactly
the set of degree $2$ covers $Z \to {\mathbb P}^1_y$ which are quotients of $S_m \to {\mathbb P}^1_y$.  
By \cite[Proposition 3]{GP05}, an application of \cite[Theorem C]{KR89}, there is an isogeny
\[\Jac{(S_m)} \sim \oplus_{i=1}^{q-1}{\Jac{(Z_{i})}}.\]

By Artin-Schreier theory, $\mu_i f(y)$ can be modified by any polynomial of the form 
$T^2- T$ for $T \in {\overline \F}_2[y]$ without changing the ${\overline \F}_2$-isomorphism class of the Artin-Schreier cover $Z_i \to {\mathbb P}^1_y$. 
Thus $Z_i$ is isomorphic to an Artin-Schreier curve with equation
$z_i^2- z_i=h_i(y)$ for some $h_i(y) \in \overline{\F}_2[y]$ with degree 
$2q_0+1={\rm max}\{(q_0+q)/q_0, q_0+1\}$.
For $1 \leq i \leq q-1$, 
the curve $Z_i$ is a $\mathbb{Z}/2$-cover of the projective line branched only at $\infty$, 
where it is totally ramified. 
Moreover, the break in the filtration of higher ramification groups in the lower numbering is at 
index ${\rm deg}(h_i(y))=2q_0+1$.
By \cite[VI.4.1]{Stich}, the genus of $Z_i$ is $q_0$.
Thus $g={\rm dim}(\Jac(S_m))=(q-1){\rm dim}(\Jac(Z_i))=q_0(q-1)$.
\end{proof} 

\begin{remark}
Consider the Artin-Schreier curve $Z_i: z_i^2- z_i = h_i(y)$
from the proof of Lemma~\ref{Lgenus}.
By \cite[Proposition 3.4]{elkinpries}, since ${\rm deg}(h_i)=2q_0+1 \equiv 1 \bmod 4$, 
the $a$-number of $Z_i$ is $q_0/2$.
Thus the $a$-number of $\oplus_{i=1}^{q-1}{\rm Jac}(Z_i)$ is $q_0(q-1)/2$, 
exactly half of the genus of $S_m$.
The fact that ${\rm Jac}(S_m)$ is isogenous to $\oplus_{i=1}^{q-1}{\rm Jac}(Z_i)$
gives little information about the $a$-number of $S_m$
since the $a$-number is not an isogeny invariant.
\end{remark}

The Hasse-Weil bound states that a (smooth, projective, connected) curve $X$ of genus $g$ defined over $\mathbb{F}_q$ must satisfy 
\[
q+1-2g\sqrt{q}\leq\#X(\mathbb{F}_q)\leq q+1+2g\sqrt{q}.
\]
A curve that meets the upper bound is called an $\mathbb{F}_q$-maximal curve.  

It is easy to check that the number of $\F_{q^2}$-points on the Suzuki curve 
is $\#S_m\left(\F_{q^2}\right)=q^2+1$ and so $S_m$ is not maximal over $\F_{q^2}$.
Analyzing powers of the eigenvalues of Frobenius shows the following.  

\begin{lemma}
The Suzuki curve $S_m$ is $\F_{q^4}$-maximal.
\end{lemma}

\begin{proof}
The $L$-polynomial of $S_m$ is $L(S_m, t)=(1+\sqrt{2q} t +q t^2)^g$ \cite[Proposition 4.3]{H92}.
This factors as $L(S_m,t)=(1 - \alpha t)^g (1 - \overline{\alpha} t)^g$ where 
$\alpha = −q_0 (1 + i)$. 
That implies that 
$\#S_m (\F_{q^4} ) = q^4 + 1 - (-q_0)^4 (\alpha^4 + \overline{\alpha}^4)g = q^4 + 1 + 2q^2g$
which shows that $S_m$ is $\F_{q^4}$-maximal.
\end{proof}

A curve which is maximal over a finite field is supersingular, in that the slopes of the Newton
polygon of its $L$-polynomial all equal $1/2$.
Thus $S_m$ is supersingular.
The supersingularity condition is equivalent to the condition that $\Jac(S_m)$ is isogenous to a product of supersingular elliptic curves.  
A supersingular curve in characteristic $2$ has $2$-rank $0$.
This implies, a priori, that the $a$-number of $S_m$ is at least one.

\subsection{Regular $1$-forms}

To compute a basis for the vector space $H^0(S_m, \Omega^1)$ of regular $1$-forms on $S_m$, 
consider the functions $h_1,h_2\in\F\left(S_m\right)$ given by: 
\begin{align*}
h_1:&=z^{2q_0} + y^{2q_0+1}, \\
h_2:&=z^{2q_0}y+h_1^{2q_0}.
\end{align*}

For any $f\in\F(S_m)$, let $v_\infty(f)$ denote the valuation of $f$ at $P_\infty$.

\begin{lemma} \label{poleorder}
The functions $y,z,h_1,h_2 \in \F(S_m)$ have no poles except at $P_\infty$ where 

\begin{align*}
v_y:&=-v_\infty(y) =q,  & v_z:&=-v_\infty(z)=q+q_0, \\
v_{h_1}:&=-v_\infty(h_1) =q+2q_0,  & v_{h_2}:&=-v_\infty(h_2)=q+2q_0+1. \\
\end{align*}

The function $\pi = h_1/h_2$ is a uniformizer at $P_\infty$.
\end{lemma}

\begin{proof} See \cite[Proposition 1.3]{HansenStich}.
\end{proof}

The function $y$ is a separating variable so $dy$ is a basis of the $1$-dimensional vector space of differential $1$-forms.  
The next lemma shows that $dy$ is regular.

\begin{lemma} \label{dy}
The differential $1$-form $dy$ satisfies 
\[
v_\infty(dy) = 2g-2 \quad \mbox{and} \quad v_P(dy) = 0 
\]
for all points $P \in S_m(\overline{\F}_q)$.
\end{lemma}

\begin{proof}
Recall that $\pi$ is a uniformizer at $P_\infty$. To take the valuation of $dy$ at $P_\infty$, 
we first rewrite $dy=f(x,y) d\pi$ for some $f(y,z) \in \F_q(y,z)$.  Note that
\begin{align*}
d\pi &= d\left(\frac{h_1}{h_2}\right)
= \frac{h_2\,dh_1 - h_1 \,dh_2}{h_2^2} 
= \frac{h_2\,y^{2q_0} - h_1 \,z^{2q_0}}{h_2^2} dy.
\end{align*}
Since $v_\infty(h_2^2) = -2(q+2q_0+1)$ and
\begin{align*}
v_\infty(h_2\,y^{2q_0} - v_{h_1} \,z^{2q_0}) &= \min \{ -2q_0v_y - v_{h_2}, -2q_0v_z - v_{h_1}\}\\
&= -2q_0v_z - v_{h_1} \\
&= -4q_0^3 - 2q_0,
\end{align*}
we see that
\begin{align*}
v_\infty(dy) &= v_\infty \left(\frac{h_2^2}{h_2^{2q_0} - h_1 z^{2q_0}} d\pi \right) \\
&= -2q + 2q_0 -2 - \left( -4q_0^3 - 2q_0 \right)  \\
&= 4q_o^3 -2q_0 -2 \\
&= 2g-2.
\end{align*}
We next show that $dy$ has no zero or pole at any affine point of $S_m$.
Note that, for any $a \in \overline{\F}_q$, the polynomial $z^q + z + a$ splits into distinct factors in $\overline{\F}_q(z)$, so there are exactly $q$ points of $S_m(\overline{\F}_q)$ 
lying over any $y_0 \in \mathbb{A}^1_y(\overline{\F}_q)$. Since
\[
\left[\overline{\F}_q(y,z):\overline{\F}_q(y)\right]=q,
\]
the $\overline{\F}_q$-Galois cover $S_m \to {\mathbb P}^1_y$ is unramified 
at all affine points of $S_m$. Consequently, for any point $P\in S_m(\overline{\F}_q)$ lying over $a\in\mathbb{A}^1_y(\overline{\F}_q)$, we see that $v_P(y-a) = 1$. 
Thus, $y-a$ is a uniformizer at $P$ and $v_P(dy)=0$, proving the proposition. 
\end{proof}

By Lemma~\ref{dy}, finding a basis for $H^0(S_m, \Omega^1)$ is equivalent to finding a basis for $L((dy))$, since $(dy) = (2g-2) P_\infty$ is the canonical divisor.  
To do this, we make use of the relations:
\begin{equation} \label{reduction}
z^2 = yh_1 +h_2, \quad h_1^{q_0} = z + y^{q_0+1}, \quad h_2^{q_0} = h_1 + zy^{q_0},
\end{equation}
which can be verified by direct substitution, and the following proposition.

\begin{proposition} \cite[Proposition 1.5]{HansenStich} \label{semithm}
Let $SG$ be the semigroup $\left\langle q,q+q_0,q+2q_0,q+2q_0+1 \right\rangle$.
Then $\#\left\{n\in SG\mid 0\leq n\leq 2g-2\right\}=g$.
\end{proposition}

We now have all the required information to find a basis of $H^0(S_m, \Omega^1)$.

\begin{proposition} \label{basis}
The following set is a basis of $H^0(S_m, \Omega^1)$:
\[
\mathcal{B}:=\left\{y^a z^b h_1^c h_2^d \, dy \, \mid \, (a,b,c,d) \in \mathcal{E} \right\}
\]
where $\mathcal{E}$ is the set of $(a,b,c,d) \subset \mathbb{Z}^4$ satisfying
\[\begin{array}{c}
0 \leq b \leq 1,\quad 0 \leq c \leq q_0-1,\quad 0 \leq d \leq q_0-1,\\
 av_y+bv_z+cv_{h_1}+dv_{h_2} \leq 2g-2. 
\end{array}\]
\end{proposition}


\begin{proof}
To prove linear independence, it suffices to prove that all 
elements in our basis have distinct valuations at $P_\infty$.  
Suppose that $y^a z^b h_1^c h_2^d \, dy \in \mathcal{B}$ and $y^{a'} z^{b'} h_1^{c'} h_2^{d'} \, dy\in\mathcal{B}$ have the same valuation at $P_\infty$; we will show they are equal. Comparing their valuations at $P_\infty$, we must have that
\begin{equation}\label{poleeq}
(a-a')v_y + (b-b')v_z + (c-c')v_{h_1} +  (d-d')v_{h_2} =0
\end{equation}

Now consider equation~(\ref{poleeq}) modulo $q_0$.  As $q_0$ divides $v_y,v_z$ and $v_{h_1}$, 
\[
 (d-d') \equiv 0 \mod q_0.
\]
As $0 \leq d,d' < q_0$, it must be the case that $d = d'$.  
Substituting $d-d'=0$ into equation~(\ref{poleeq}) and reducing modulo $2q_0$ yields that 
\[
(b-b')q_0  \equiv 0 \mod 2q_0.
\]
However, as $0 \leq b,b' \leq 1$, it must also be the case that $b = b'$.  Simplifying~(\ref{poleeq}) and reducing modulo $q = 2q_o^2$ yields that
\[
(c-c')(q-2q_0) = (c-c')2q_0 \equiv 0 \mod 2q_0^2 .
\]
Since $0 \leq c,c' \leq q_0-1$, we find that $c = c'$; so $a= a'$ as well. 

We claim that the above set also spans $L\left((dy)\right)$.  Clearly the valuations at $P_{\infty}$ of 
\begin{align*}
\left\{y^az^bh_1^ch_2^c\mid(a-a')v_y + (b-b')v_z + (c-c')v_{h_1} +  (d-d')v_{h_2} \leq 2g-2 \right\}
\end{align*}
are equal to $\left\{n\in SG\mid 0\leq n\leq 2g-2\right\}$, which is a set of size $g$ by Proposition~\ref{semithm}. Rewriting elements of the above set in terms of our basis will not change their valuation at $P_\infty$. Thus we can use the relations in equation~(\ref{reduction}) to see that $\mathcal{B}$ also contains an element for each of the $g$ possible valuations at $P_{\infty}$. By the previous paragraphs, each valuation occurs exactly once. By Riemann-Roch, $\ell\left((dy)\right)= g$, so $\mathcal{B}$ is a basis. 
\end{proof}

\subsection{Action of the Cartier operator}\label{Holley}

In characteristic $2$, the Cartier operator $\car$ acts on differential $1$-forms according to the following properties: 
(see e.g., \cite[Section 2.2.5]{TVN}). 

\begin{enumerate}
\item $\car$ is $1/2$-linear; i.e., $\car$ is additive and $\car{(f^2\omega)}=f\car{(\omega)}$.

\item$\car{(y^j\,dy)}=
\begin{cases}
0, & \textrm{if }j\not\equiv 1 \mod 2\\
y^{e-1}\,dy&\textrm{if }j=2e-1.
\end{cases}$
\item $\car{(\omega)}=0$ if and only if $\omega$ is exact; i.e., if and only if $\omega=df$ for some $f \in \F_q\left(S_m\right)$. 
\item $\car{(\omega)}=\omega$ if and only if $\omega=df/f$ for some $f \in\F_q\left(S_m\right)$.
\end{enumerate}

Any 1-form $\omega \in H^0(S_m, \Omega^1)$ can be written in the form $\omega=(f^2+g^2y) dy$, as $\ch(\F_q)=2$. 
Then 
\begin{equation} \label{Ecartier}
\car{((f^2+g^2y) dy)}=g \, dy.
\end{equation}
By these properties, it is clear that
\begin{equation}\label{eq:evenc}
\car(y^{2e_1+r_1}z^{2e_2+r_2}h_1^{2e_3+r_3}h_2^{2e_4+r_4}\,dy)=y^{e_1}z^{e_2}h_1^{e_3}h_2^{e_4}\car(y^{r_1}z^{r_2}h_1^{r_3}h_2^{r_4} \,dy).
\end{equation}
 
Hence to compute the action of $\car$ on $H^0(S_m, \Omega^1)$, we need only compute $\car$ on the 16 monomials in $y,z,h_1,h_2$ of degree less than or equal to one in each variable. The table below shows this action, 
where each $\car(f\, dy)$ is written in terms of the original basis using the curve equation~(\ref{Suzuki}).

\[\displaystyle \begin{array}{|l|l|}
\hline
f&\car{(f dy)}\\
\hline
\hline
1&0\\ 
\hline
y&dy\\
\hline
z&y^{q_0/2}\, dy\\
\hline
h_1&y^{q_0}\, dy\\
\hline
h_2&\left((yh_1)^{q_0/2}+h_2\right)\, dy\\
\hline
yz&h_1^{q_0/2}\,dy\\
\hline
yh_1&\left((yh_1)^{q_0/2}+h_2\right)\, dy\\
\hline
zh_1&(yh_2)^{q_0/2}\, dy\\
\hline
zh_2&(h_1h_2)^{q_0/2}\, dy\\
\hline
h_1h_2&\left(h_1+zy^{q_0}\right)\, dy\\
\hline
yzh_1&\left(y^{q_0/2}z+(h_1h_2)^{q_0/2}\right)\, dy\\
\hline
yzh_2&\left(zh_1^{q_0/2}+y^{q_0/2+1}h_2^{q_0/2}\right)\, dy\\
\hline
zh_1h_2&\left(zy^{q_0/2}h_2^{q_0/2}+h_1^{q_0/2+1}\right)\, dy\\
\hline
yh_1h_2&\left((yh_1)^{q_0/2}z+h_2^{q_0/2}z\right)\, dy\\
\hline
yzh_1h_2&\left(y^{q_0/2}h_2+zh_1^{q_0/2}h_2^{q_02}\right)\, dy\\
\hline
\end{array}\]


\begin{example}
We illustrate the computation for $zh_1h_2\, dy$.
Direct computation yields
\begin{align*}
\car{(zh_1h_2\, dy)}&=\car{\left(zh_1\left(yz^{2q_0}+h_1^{2q_0}\right)\, dy\right)}\\
&=z^{q_0}\car{\left(zyh_1\, dy\right)}+h_1^{q_0}\car{\left(zh_1\, dy\right)}\\
&=\left(y^{q_0/2}z^{q_0+1}+h_1^{q_0/2}h_2^{q_0/2}z^{q_0}+h^{q_0/2}h_1^{q_0/2}h_2^{q_0/2}\right)\, dy.
\end{align*}
To write this expression in terms of the original basis, we identify the monomials with the highest pole order at infinity. Since
\[
v_\infty(h_1^{q_0/2}h_2^{q_0/2}z^{q_0})=v_\infty(h^{q_0/2}h_1^{q_0/2}h_2^{q_0/2})=-4q_0^3-3q_0^2-q_0/2<-\left(2g-2\right),
\]
these two terms may be simplified. Using Section~\ref{Colin},
\begin{align*}
h_1^{q_0/2}h_2^{q_0/2}z^{q_0}+&h^{q_0/2}h_1^{q_0/2}h_2^{q_0/2}\\
&=h_1^{q_0/2}h_2^{q_0/2}\left(y^{q_0/2}h_1^{q_0/2}+h_2^{q_0/2}\right)+y^{q_0/2}h_1^{q_0}h_2^{q_0/2}
=h_1^{q_0/2}h_2^{q_0}.
\end{align*}
The final expression follows by rewriting $z^{q_0+1}$ and $h_2^{q_0}$ in terms of lower order basis elements using equations~(\ref{reduction}).
\end{example}


\begin{remark}
To compute $\car{(\omega)}$ for a general element $\omega\in\mathcal{B}$, simply apply equation~(\ref{eq:evenc}) and use the table above; in nearly all cases the direct result will again be in terms of the basis $\mathcal{B}$. The only exception is when $\omega=zh_1^{q_0-1}h_2\,dy$. In this case we have:
\begin{align*}
\car{\left(h_1^{q_0-2}\cdot zh_1h_2\, dy\right)}&=h_1^{q_0/2-1}\car{\left(zh_1h_2\right)}\\
&=\left(zy^{q_0/2}h_1^{q_0/2-1}h_2^{q_0/2}+h_1^{q_0}\right)\, dy.
\end{align*}
Using equations~(\ref{reduction}), one can obtain an expression in terms of the original basis. 
\end{remark}

\section{The $a$-number for Suzuki curves}\label{Derek}

\subsection{A closed form formula for the $a$-number} 
We now have the tools to compute $a(m)$. The calculation amounts to counting lattice points in polytopes in $\mathbb{R}^3$, which is a hard problem in general. In our case, however, the values $v_y$, $v_{h_1}$, and $v_{h_2}$ are so similar that the polytopes in question are nearly regular; this makes our counting problem much easier.

\begin{theorem} \label{maintheorem}
Let $a(m)$ and $g(m)$ be the $a$-number and genus of $S_m$ respectively. Then
\[
a(m)=\frac{q_0(q_0+1)(2q_0+1)}{6}.
\]
In particular,
\[
\frac{1}{6}<\frac{a(m)}{g(m)}<\frac{1}{6}+\frac{1}{2^{m+1}}.
\]
\end{theorem}

\begin{proof}
Recall from Section~\ref{preliminaries} that $a(m)$ is the dimension of the kernel of $\car$
on $\Hnot$.
By equation~(\ref{Ecartier}), $a(m)$ is the dimension of the vector space of regular differentials of the form $f^2\,dy$. Since $f^2$ can have a pole only at $P_{\infty}$, and since the order of the pole can be at most $2g-2$, we see that $a(m)=\ell((g-1)P_{\infty})$. Moreover, squaring is a homomorphism, so
\[
\ell((g-1)P_{\infty})=\#\left\{\omega\in\mathcal{B}\mid v_{\infty}(\omega)\leq g-1\right\}.
\]
By Section~\ref{Colin}, this number is exactly
\[
\#\left\{(a,b,c,d) \in \mathcal{E} \mid v_ya+v_zb+v_{h_1}c+v_{h_2}d\leq g-1
\right\};
\]
here we use the notation $\mathcal{E}$ as we did in Proposition~\ref{basis}.
Recall that $b \in \{0,1\}$. 
When $b = 0$, we must count $\{(a,c,d)\in\mathbb{N}^3\mid a+c+d\leq q_0-1\}$. This follows from the fact that that
\[
q_0-1=\frac{g-1}{v_{h_2}}<\frac{g-1}{v_{h_1}}<\frac{g-1}{v_y}<q_0.
\]
For $b = 1$, we must count $\{(a,c,d)\in\mathbb{N}^3\mid a+c+d\leq q_0-2\}$ since
\[
q_0-2<\frac{g-1-v_z}{v_{h_2}}<\frac{g-1-v_z}{v_{h_1}}<\frac{g-1-v_z}{v_y}<q_0-1.
\]
Using these two facts, we obtain
\begin{align*}
a(m)
=&\#\left\{(a,c,d)\in\mathbb{N}^3\mid a+c+d\leq q_0-1\right\} \\ 
 & \quad + \#\left\{(a,c,d)\in\mathbb{N}^3\mid a+c+d\leq q_0-2\right\}\\
=&\sum_{i=2}^{q_0+1}{\binom{i}{2}}+\sum_{i=2}^{q_0}{\binom{i}{2}}\\
=&1+\sum_{i=2}^{q_0}{\left(\binom{i+1}{2}+\binom{i}{2}\right)}\\
=&\sum_{i=1}^{q_0}{i^2},
\end{align*}
as desired.

To prove the second statement, simply note that
\[
\frac{1}{6}<\frac{q_0(q_0+1)(2q_0+1)}{6q_0(q-1)}=\frac{1}{6}\cdot\frac{q_0^2+\frac{3}{2}q_0+\frac{1}{2}}{q_0^2-\frac{1}{2}}<\frac{1}{6}\left(1+\frac{3}{q_0}\right).
\]
\end{proof}

\subsection{Open questions} \label{Corollaries}

Here are two open questions about $\Jac(S_m)$.

\begin{question} \label{Q1}
What is the decomposition of $\Jac(S_m)$ into indecomposable principally polarized abelian varieties?
\end{question}

Theorem~\ref{maintheorem} gives partial information about Question~\ref{Q1}, namely an upper bound on the number of factors appearing in the decomposition, because of the following fact.

\begin{lemma} \label{decompose}
Suppose $A$ is a principally polarized abelian variety with $p$-rank $0$ and $a$-number $a$.
If $A$ decomposes as the direct sum of $t$ principally polarized abelian varieties, then $t \leq a$.
\end{lemma}

\begin{proof}
Write $A \simeq \oplus_{i=1}^t A_i$ where each $A_i$ is a principally polarized abelian variety.
For $1 \leq i \leq t$, consider the $p$-torsion group scheme $A_i[p]$.  
The $a$-number of $A_i[p]$ is at least $1$ since its $p$-rank is $0$.
Thus the $a$-number of $A$ is at least $t$.
\end{proof}

To state the second question, we need some more notation.

The Ekedahl-Oort type of a principally polarized abelian variety $A$ over $k$ 
is defined by the interaction between 
the Frobenius $F$ and Verschiebung $V$ operators on the $p$-torsion group scheme $A[p]$.
It determines the isomorphism class of $A[p]$ and its invariants such as the $a$-number.
To define the Ekedahl-Oort type, 
recall that the isomorphism class of a symmetric $BT_1$ group scheme ${\mathbb G}$ over $k$
can be encapsulated into combinatorial data.
This topic can be found in~\cite{O:strat}.
If ${\mathbb G}$ has rank $p^{2g}$, 
then there is a \emph{final filtration} $N_1 \subset N_2 \subset \cdots \subset N_{2g}$ 
of ${\mathbb G}$ as a $k$-vector space
which is stable under the action of $V$ and $F^{-1}$ such that $i=\dim{(N_i)}$.
The {\it Ekedahl-Oort type} of ${\mathbb G}$, also called the {\it final type},
is $\nu=[\nu_1, \ldots, \nu_r]$ where ${\nu_i}=\dim{(V(N_i))}$.
The Ekedahl-Oort type of ${\mathbb G}$ is canonical, even if the final filtration is not.

There is a restriction $\nu_i \leq \nu_{i+1} \leq \nu_i +1$ on the final type.
Moreover, all sequences satisfying this restriction occur.   
This implies that there are $2^g$ isomorphism classes of symmetric $BT_1$ group schemes of rank $p^{2g}$.  The $p$-rank is $\max{\{i \mid \nu_i=i\}}$ and the $a$-number equals $g-\nu_g$. 

\begin{question} \label{Q2}
What is the Ekedahl-Oort type of $\Jac(S_m)[2]$?
Equivalently, what is the covariant Dieudonn\'e module of $\Jac(S_m)[2]$?
\end{question}

Theorem~\ref{maintheorem} gives partial information about Question~\ref{Q2}, by limiting the possible final types.
For the group scheme $\Jac{\left(S_m\right)}[2]$, the Ekedahl-Oort type satisfies that
$\nu_1 = 0$ and $\nu_g = q_0(10q_0+7)(q_0-1)/6$. 
In particular, $\Jac(S_m)$ is not superspecial since $a(m) \not = g(m)$.
This implies that $\Jac(S_m)$ is not isomorphic to the product of supersingular elliptic curves;
it is only isogenous to the product of supersingular elliptic curves.

In the next example, we give some more information about the Ekedahl-Oort type of $\Jac(S_1)[2]$ (the case $m=1$).

\begin{example}
If $m=1$ then $q_0=2$, $q=8$, and $g=14$.
By Section~\ref{Holley}, the image of $\car$ on $\Hnot$ is spanned by 
the nine $1$-forms 
\[\{dy, ydy, h_1dy, y^2dy, yh_1dy, yh_2, (z+y^3)dy, h_1h_2dy, y^2zdy \}.\]
The image of $\car^2$ on $\Hnot$ is spanned by 
the four $1$-forms 
\[\{dy, y^2dy, (z+y^3)dy, (h_1+y^2z)dy\}.\]
Also $\car^3$ trivializes $\Hnot$.
Thus $\nu_1=\nu_2=\nu_3=\nu_4=0$, and $\nu_9=4$, and $\nu_{14}=9$.
The combinatorial restrictions on the final type imply that 
$\nu_{10}=5$, $\nu_{11}=6$, $\nu_{12}=7$, and $\nu_{13}=8$.
This leaves only five possibilities for the final type, and thus for the isomorphism 
class of $\Jac(S_1)[2]$.
\end{example}

\bibliographystyle{amsplain}
\bibliography{summergradSuzuki}

Holley Friedlander\\
University of Massachusetts, Amherst\\
Amherst, MA 01003\\
holleyf@math.umass.edu

\bigskip

Derek Garton\\
University of Wisconsin--Madison\\
Madison, WI 53706\\
garton@math.wisc.edu

\bigskip

Beth Malmskog\\
Wesleyan University\\
Middletown, CT 06457\\
emalmskog@wesleyan.edu

\bigskip

Rachel Pries\\
Colorado State University\\
Fort Collins, CO 80521\\
pries@math.colostate.edu

\bigskip

Colin Weir\\
University of Calgary\\
Calgary, AB, Canada\\
cjweir@ucalgary.ca

\end{document}